\documentclass{amsart}
\newtheorem{theorem}{Theorem}[section]

\newtheorem{proposition}[theorem]{Proposition}
\theoremstyle{definition}
\newtheorem{definition}[theorem]{Definition}

\theoremstyle{remark}
\newtheorem{remark}[theorem]{Remark}
\numberwithin{equation}{section}
\usepackage{amscd,amssymb,amsmath,amsfonts,latexsym,enumerate}
\usepackage[arrow,matrix,graph,frame,poly,arc,tips]{xy}
\usepackage [ latin 1]{ inputenc }
\theoremstyle{plain}
\theoremstyle{definition}

\numberwithin{equation}{section}

\newcommand{\R}{\mathbb{R}}

\def\mm{\overline{M}}

\date{\today}

\begin{document}
\title[On the Topology of the  Fano Surface ]{Remarks On the Topology  of the  Fano Surface}
\author{Alberto Collino}
\address{Dipartimento di Matematica, Universit\'a di Torino, via Carlo Alberto 10, 10123 Torino, Italy}
\begin{abstract}
We analize the semistable degeneration of the smooth  Fano surface $F$ when the cubic threefold becomes the Segre primal. This gives    an explicit   topological  decomposition  for  $F.$ The components are  surfaces with boundary, they come from  two of the simplest  line arrangements in the projective  plane. The combinatorics of their intersections  is described by   a Kneser graph. The decomposition is used to decide   that the  Fano surface is not an   an Eilenberg Mac-Lane  $K(\pi, 1)$ space, this was the question that prompted us to look into the matter.  
\end{abstract}

\maketitle

\section {Introduction}
\label{sect:intro}
Consider  a  cubic threefold $X \subset \mathbb P^4 ,$  the     Fano variety $F(X)$   is the locus  inside the  Grassmannian $G(1,4)$  made of  the lines on $X$.  If $X$ is smooth then $F(X)$ is a non singular surface whose Hodge structure is classically well understood \cite{CG}.  We  provide     an explicit   topological  decomposition  for   $F(X)$  which   we use to show that the  surface is not an   an Eilenberg Mac-Lane  $K(\pi, 1)$ space, in fact  $\pi_2(F)$  is not of torsion. This  gives   a negative answer to  a question asked by P.Pirola as an afterthought to \cite {CPN}.  Our result  should be compared with   Roulleau's ones,   he    proved in \cite{ROU}    that there are   ball quotients  sitting  as     Zariski open subsets  in  the    Fano surfaces of the Fermat cubics.\paragraph{} 

The decomposition that we use is obtained by means of a semistable degeneration of the Fano surface, one    which is associated  with  the degeneration of  the cubic threefold when it moves in a one-parameter family and specializes to  the Segre primal. This   threefold has ten nodes, the largest  number of  isolated singularities that a cubic threefold can have.    The  degeneration for the Fano surface is also maximal, the singular fibre splits in $21$ components, all rational surfaces.  The nature of the components and the  combinatorics   of their intersections  is dictated by the geometry of $\mm_{0,6}$ (the moduli spaces of stable $6$-pointed rational curves), this is so  because $\mm_{0,6}$  is a natural desingularization of  the   Segre cubic.  The     limiting Mixed Hodge structure on $H^1$ of  a nonsingular fiber  $F_t$  is totally degenerated, i.e. it is   of Hodge-Tate type.   The Clemens-Schmid exact sequence can  then be  used to prove a certain  injectivity statement. This  is the main   tool   in our       proof  that  $\pi_2(F)$  is not of torsion,  given in the last  section.    Although  the statement is purely topological it is apparent that our argument  depends   crucially  on fundamental theorems   of   Hodge theory, cf. \cite {GrS}. 
\newpage 
\section{  \label{degenerazione} A topological decomposition of the Fano surface}
\subsection{ The degeneration method}\paragraph{}
The  Fano surface  $F$ of  a smooth cubic 3-folds is non singular and then any two of them are homeomorphic. The methods   of  \cite{CG} and  \cite{HC}, see also   \cite{Pe},  work  quite effectively  for a pencil of cubic 3-folds with central fibre the Segre primal,   the outcome is   an explicit  topological decomposition for  $F$ .\subsubsection{The general setting of the  degeneration method for surfaces} \paragraph{}
Consider $q:\, \mathcal Y \to \Delta $,  a proper flat holomorphic map from a smooth, analytic space to the unit disk.  We say that this is  a degeneration of surfaces  if   $Y_t := q^{-1}  (t) $ is a complete variety, of dimension $2$   and  moreover $Y_t$  is non singular for $t\neq 0$. The degeneration is   stable if the central fibre $Y_0$ is a reduced divisor with global normal crossing singularities. 
	A stable  degeneration contains rich information relating the singular and the smooth fibre,  our immediate   aim  
is to recall  the surgery procedure. This is a gluing process which builds  topologically the smooth fibre  by means of    the decomposition in irreducible components  of   the singular fibre. \paragraph{}
Let $\mathcal D \subset Y_0 $ be  the singular divisor on the central fibre,  it  is the union of the double curves, where two components meet,  and let $\mathcal T \subset \mathcal D$ be the set of triple points, the locus  of intersection of three components.  
  The   surgery process  depends   on   the existence  of Clemens'   collapsing function  $c:\, \mathcal Y \to  Y_0 \, $, which maps the total space onto the central fibre. The restriction $c_t: \, Y_t \to Y_0$ has the properties \begin {enumerate}  \item  $ c^{-1} (y)  = point \, , $  if $ y \in Y_0 \setminus \mathcal D$  \item  $ c^{-1} (y)  = S^1  \, , $  if $ y \in \mathcal D \setminus \mathcal T$  \item  $ c^{-1} (y)  = S^1 \times S^1  \, , $  if $ y \in \mathcal T \, .$ \end {enumerate} The gluing takes place  along the  boundary manifolds.
  
\subsubsection{ The open complement and the boundary manifold} \paragraph{}   Let $S$ be a non singular surface and $C:= \cup C_j \subset S$ be a connected  normal crossing divisor with non singular components.   In our situation  the components of $C$   are going to be  lines and  they intersect  is at most one point.  Consider the complement  $$S^{0}(C):=  S \setminus C\, ,$$ take  an open regular normal neighbourhood $N^0$ of $C$ and look at the boundary $$M(C,S):= \partial N^0\, .$$  This $M$ is our   boundary manifold, of real dimension $3$. Since    $M$   is  the boundary also of \begin{equation}  \label{withbd}   \bar S(C):=S \setminus N^0 \, ,     \end{equation}  which is homotopically equivalent to $S^{0}(C)$,  then Lefschetz duality and excision yield: \begin {equation} \label{omologiauno} H_1( S^{0}(C))\simeq H_1(S\setminus N^{0}) \simeq H^{3}(S\setminus N^{0}, M)\simeq H^{3}(S, C)\, . \end {equation} The fibre of the projection  $g:\,  M(C,S) \to C$ is  $S^1$, for a smooth point of $C$, while  it is homeomorphic to $S^1 \times S^1$ over a  node, cf. \cite {Pe}. More precisely  $ M(C,S)$  is homotopic  to $(C \setminus T) \times  S^1 $ here $T$ is the non empty set of double points  of $C$. The torus comes from the restriction of the $S^1$ fibration to the boundary of a disc around a point in $T$. In   \cite {CS} or \cite {HI} one finds   detailed descriptions of the properties of boundary manifold for lines arrangements.  The surfaces that appear   in our  decomposition of  $F$  come from   the simplest instances of    arrangements, they are    the complement $B^0$ of four general lines in $\mathbb P^2$   and   the complement $D^0$ of the ten  $-1$  lines in the del Pezzo surface of degree $5$.  One has  
 \begin {equation} \label{omologiaarra}
 \pi_1 (B^0) =H_1 (B^0)=\mathbb Z^3  {\text {  \,and \, }}  H_1 (D^0)=\mathbb Z^5   \, .\, \end {equation} 
\subsubsection{ The topological decomposition} \label{decomposizione} \paragraph{}
The central fibre of our degeneration $\mathcal Y$  is  a normal crossing divisor,  $Y_0 = \cup Y_n$. On a  component  the  double curve $  D_n:= \mathcal D \cap Y_n$ is singular precisely  along   the set of triple points $T_n:= \mathcal T \cap V_n$.  Consider now  the manifolds   $\bar Y_n(D_n)$  and their boundaries   $M(D_n,Y_n) $.  One has:    \begin{theorem} {\bf{ [Clemens]}} $Y_t$  is homeomorphic to   the topological  space built by gluing the disjoint union of the manifolds  $\bar Y_n(D_n)$
along their boundaries so that   there is  identification of     fibres  over the same point in $\mathcal D $ for  the different   projections  $ M(D_n,Y_n) \to D_n $.  \end{theorem}
There are now two compatible  Mayer-Vietoris spectral sequences on the stage, one abutting to  $H^{m}(Y_0)$ and the second to  $H^{m}(Y_t)\, ,$ see  (2.5) in \cite {Pe};  the contravariant morphism \begin{equation}  \label {controvariant} c^{\ast} : H^{m}(Y_0) \to H^{m}(Y_t)  \end{equation}
is induced by the    natural map among  them.
In particular using (\ref {omologiauno})  we note that for a component $S$ of $Y_0$ as above,  the map $H_1( S^{0}(D_S)) \to H_1( Y_t) $   can be factorized as   
\begin{equation}   H_1( S^{0}(D_S)) \simeq  H^{3}(S, D_S) \to H^{3}(Y_0) \to H^{3}(Y_t)\simeq  H_1( Y_t)  \label{fattorizzo} \end{equation} 
In our special situation,   the case of the degeneration of the Fano surface that we consider next,   we shall check  that both   arrows  $H^{3}(S, D_S) \to H^{3}(Y_0)$ and $H^{3}(Y_0) \to H^{3}(Y_t)$ are injective.
\subsection{ The Segre primal and its Fano surface} \subsubsection{ The Segre primal} \paragraph{} Consider the isomorphism of   $\mathbb P^4$ with the diagonal  hyperplane in  $  \mathbb P^5$ of    equation  $\sum x_i= 0 \,.$  The Segre primal  $ \mathcal S $ is  the hypersurface in $ \mathbb P^4$  given  by the  intersection of the hyperplane with the Fermat cubic fourfold. 
On  $\mathcal S$ there  $10$ double points, 
and therefore one has also  $45$ distinguished lines, each one spanned by a couple of  double points. It turns out that there are
 $15$ planes on $\mathcal S,$ any  one of them  passes through four of the nodes and contains  six distinguished lines \cite { SEGRE}. 
The  geometry of  $\mathcal S$  has been investigated  recently in  \cite {GW}, there  a proof of the non rationality of the generic smooth cubic threefold is given  by making use   of Alexeev's theory \cite{VA}. The proof    depends  on certain  intriguing  connections which relate  the combinatorics of the configurations  above  with the classification of  matroids.
  \subsubsection{ The Fano   surface $ F(\mathcal S)$ }\paragraph{} Our interest in the Segre primal  is due to the structure of    $F(\mathcal S)$, the Fano surface of the lines which lie in $\mathcal S$. It turns out that $F(\mathcal S)$ is made of $21$ components. The description of  what the components are,  of  the combinatorics  of their intersection,
and of the action of  $S _6 $ on  $F(\mathcal S)$  is   easier  to understand and simpler  to state by making use of the property  that  $\mathcal S $ can be described as the image of the universal stable line with $5$ marked points on it, which image is obtained  by contracting    the tails of the reducible curves.  
\subsubsection{Reminder on moduli spaces of stable ${n}$-pointed rational curves}\paragraph{}
A stable $n$-pointed rational curve is a connected curve   $C$  whose irreducible  components are projective
lines, with the property that the singularities are ordinary nodes and with the choice of  $n$ distinct marks which are smooth points of $C$, such that every irreducible  component
has at least three special points. Here special point means a mark or a node. 
	The  moduli space for stable $n$-pointed rational curves $\mm_{0,n}$ as a set is the family of isomorphism classes of  such curves. Mumford and Knudsen proved that it  has the structure of  a smooth irreducible projective variety of dimension $n-3$.  Inside $\mm_{0,n}$ we find   as an open subset the moduli space of smooth stable irreducible  pointed curves  $M_{0,n}$. There is a stratification of $\mm_{0,n} \setminus  M_{0,n}$  by topological type of the dual graph of the curve: a \emph {codimension 1-stratum} is an irreducible component of the locus of points of $\mm _{0,n}$ having at least one node. The generic element of a codimension $1$-stratum has two irreducible components, with some $J \subseteq  \{1, \ldots, n\} $, $2 \leq |J|\leq n-2$, giving the marked points on one component, and $J^c$ giving the marked points on the other.
 The resulting divisor is called a {boundary divisor}, and it is  denoted by $\Delta_J$.   We write \begin {equation} \Delta := \partial \mm _{0,5}= \cup \Delta_{i,j}\, ,{\text { so it is  \,\,} } \, M_{0,5}  = \mm _{0,5} \setminus \Delta \, . \end  {equation} 
\subsubsection {    $\pi_1(M_{0,5})$  }  \paragraph{}
A  presentation for   $\pi_1(M_{0,5})$   was already known to Picard, cf. \cite{YY}. 
Let here $ x_1 , x_2 ,  x_3$ be homogeneous coordinates for $\mathbb P^2$,
set $x_4=0 \, ,$ and write 
$S(ij):= \{ x_i = x_j \} \subset \mathbb P^{2},\,\  \{ i ,j \} \subset \{ 1.. 4\} $,  $S(ijk):= \{ x_i = x_j = x_k \} \subset \mathbb P^{2},\,\,\,  \{ i ,j , k\} \subset \{ 1.. 4\} .$
The blow-up of $\mathbb P^2$ at the four points $S(ijk)$ gives a del Pezzo surface $D$ of degree five.
We have ten $(-1)$  lines on $D$, which are  $S(ij)^{b}$ , the proper transform of $S(ij)$,
and $S(ijk)^{b}$, the exceptional line over $S(ijk)\, .$  It  is well known that   $D \simeq   \mm_{0,5 }$,
the identification   is seen by realizing  that   $ D$ is  the universal line over $\mm_{0,4}$, the fibres being the transforms of the   conics  through the four base points $ S(ijk)\,.$ It is
 $$ S^b(ij) = \Delta_{k,l}  \,  ,  \,\,    {\text { and   }} \,\,\,\,  S^b (ijk) = \Delta_{l,5} \,\,\,\,{\text {where\,\,}}\{ i ,j , k,l \} = \{ 1.. 4\} \, .$$

The fundamental group $\pi_1 ( M_{0,5}) $ is generated by ten  loops, each one a normal loop around one of the ten $-1$  lines. Six of the loops can be written as $ \sigma_ {i,j}$ to indicate that they are the normal loops around $S^b(ij)$. The loop  around the exceptional line $S^b (ijk)$ is  then homotopic to   the product  $\sigma_{i,j} \sigma_ {i,k} \sigma_{j,k}$.  The commuting relations written in \cite{YY}  say that a  loop around an exceptional line commutes with the loops around the other  lines which  meet it.     There is the  further relation
\begin{equation}\label {eq:dprelazione}  \sigma_{12}\sigma_{13}\sigma_{23}\sigma_{14}\sigma_{24}\sigma_{34}=1 \, .\end{equation} We have therefore:\begin {equation}   H_1(M _{0,5}, \mathbb Z ) \simeq \mathbb Z ^{\oplus 5 }.  \end{equation}
\begin{remark}
It is known  that $M _{0,5}$ is a   $K(\pi, 1)$ space, as it can be seen using the long exact sequence of  homotopy  for  the fibration $M _{0,5} \to M _{0,4}  $ . \end{remark}

\subsubsection{Irreducibility of the representation of $S_5$ on $H^1(M _{0,5}, \mathbb Q)$  }\paragraph{}
The symmetric group $S_5$ acts on $\mm _{0,5} $ and in particular it acts on the  ten  dimensional  vector space $B$ with basis $\Delta_{i,j}$.  Let $ V$ be the standard $4$ dimensional representation of  $S_5$ sitting inside the tautological permutation representation $T\, .$ Now  $Sym ^{2} T\  \simeq T \oplus Sym ^{2}V \, ,$ where the copy of $T$ comes from the squares. We have then   $B \simeq Sym ^{2} V$, so by   (3.2) in \cite {FH}     it is     \begin{equation}  B= U \oplus V \oplus W \, ,\end{equation} here $U$  is  the trivial representation  and $W$ is irreducible of dimension $5$. \newline
We use this   decomposition  to determine  the action of $S_5$ on $H^1(M _{0,5})$ and $H^2(\mm _{0,5})\, .$  The cohomology sequence with compact support of  the couple $(\mm _{0,5} , \Delta)$ is : \begin{equation}   \dots   \to H^2(\mm _{0,5})  \xrightarrow{j} H^{2 } (\Delta ) \xrightarrow{} H^{3}_c (M _{0,5}) \xrightarrow{} H^{3}(\mm _{0,5}) \xrightarrow{}  \dots  \, \, .  \end {equation}  
and dually , cf.  \cite {HF}, one has the  sequence 
 \begin{equation}   \dots   \xleftarrow{}  H^2(\mm _{0,5})  \xleftarrow{} H^{2 } _{\Delta } (\mm _{0,5} ) \xleftarrow{} H^{1} (M _{0,5})\xleftarrow{} \dots  \, \, .  \end {equation}  Our surface  $\mm _{0,5}$ is   rational and the divisors $\Delta_{i,j}$ generate $H^{2 }\, ,$ then 
the  last sequence yields: 
\begin{equation}   0  \xleftarrow{}   H^2(\mm _{0,5})  \xleftarrow{}   B \xleftarrow{}   H^1 (M _{0,5}) \xleftarrow{}  0 \, \, .  \end {equation}   
Both spaces  $H^2(\mm _{0,5})$ and $H^1 (M _{0,5}) $ are  of dimension $5$,  but  $H^2(\mm _{0,5})$ contains an invariant one dimensional subspace, the one which is generated by the class of the canonical divisor.  This proves 
\begin{proposition} \label {irriducibile} $H^1 (M _{0,5},\mathbb Q)) $ is an irreducible representation of $S_5$\,. 
\end {proposition} 
 \subsubsection{ A birational map from $\mm _{0,6}$ to the Segre primal  $\mathcal S$  } \paragraph {}
  The moduli space of stable $6$-pointed rational curves      can be mapped inside  $\mathbb P^4$
birationally (this is  a special case of a Kapranov's theorem)  and the image  is isomorphic to $\mathcal S$.   This  map $\mm_{0,6}  \to \mathcal S $  contracts the $10$  boundary divisors $\Delta_J$ , $|J|= 3 $ to the $10$ nodes on $\mathcal S$,  the $15$  boundary divisors $\Delta_J$ , $|J|= 2 $ are mapped to the $15$ planes on $\mathcal S$.\newline \indent Recall next  the description of $ \mm _{0,6}$   as the universal curve. This amounts exactly to choose a label $i$ and then to  take  the forgetful  map  $\phi_{i}:  \mm _{0,6} \to \mm_{0,5 }$ given by dropping the $i^{th}$ marking.  The contraction $\mm_{0,6}  \to \mathcal S $ maps the fibres  of the universal curve  to lines on $\mathcal S$.  
 \subsubsection{ The projective construction of   $\mathcal S$  } \paragraph{}
It is classically known that   ${\mathcal S}$ is the image of  the rational map  $\mathbb{P}^3 \dasharrow \mathbb{P}^4$ given by  the linear system of quadric surfaces passing through  five points $p_i$ in general position. This fact is the key   to  understand  Kapranov's contruction.
  So one blows up first  $\mathbb{P}^3$ at the five points and then blows down the proper
transforms $\ell_{ij}$ of the ten lines  which  join $p_i$ and $p_j$.  We write  \label{ tenpoints} $p_{i,j} $ for the node on $\mathcal S$  image of  $\ell_{ij}$.
Ten of the planes in ${\mathcal S}$ are the
proper transforms of the planes in $\mathbb{P}^3$ containing   three base 
points. The remaining five  are the  images of the exceptional divisors in the blowup.  Each of the
planes contains four  nodes.
Beside the lines contained in the planes, 
${\mathcal S}$ contains six  two dimensional families of lines,  five of them   are the families of the proper transforms of the lines through the chosen points,  and one parametrizes
the images of the  twisted cubics through all of them. 
\subsubsection{The components of   $F({\mathcal S})$ } \paragraph{}
This  explains    why  the Fano surface $F({\mathcal S})$  has $21$ components, 
$15$ of them are planes and $6$ of them are del Pezzo surfaces of degree 5.  The $6$ del Pezzo surfaces, $D(i)\subset F(\mathcal S)$,  are the result of a different choice for  the label $i$ of the  forgetful map  $\phi{_i}: \mm _{0,6} \to \mm_{0,5 }$  considered  as the universal curve  fibration.
The $15$ planes $P(j,k) \subset F(\mathcal S)$ can be understood as the   family of lines sitting inside  each one of the $15$ planes on $\mathcal S$,  which planes on $\mathcal S$ are   the images of the  boundary divisors $\Delta_{i,j} $ on    $ \mm _{0,6} \, .$ The outcome is
$$ \mathcal F(S)  =   (\bigcup_{i=1..6}  D(i) ) \, \bigcup \, \,( \bigcup_{ \{j,k\} \subset \{1..6\} }P(j,k)\,\,\,\,) \,\,\,\,\, .$$
	On $\mathcal S$ there are    $45$ distinguished lines,  spanned by   couple of  nodes of $\mathcal S$.  We write them $L[(i,j),(k,l)]$,  which we identify with points in the Grassmann variety, i.e.  $$L[(i,j),(k,l)]:= P(i,j)\cap  P(k,l) \in   F(\mathcal S) \subset G(1,4)\, .$$  A  further consequence of the description of  $\mathcal S$ as the image of $ \mm _{0,6} $ yields       $$ P(j,k)\cap  P(l,m) = \emptyset \,\,\, {\text {if}} \, \, \{j,k\} \cap  \{l,m\}  \not=  \emptyset .$$  
\indent  With our notation    $D(m)$ is a copy of $ \mm_{0,5 }$,  we write 
$R(m, [i,j])$  to   represent  the boundary divisor  $\Delta _{i,j}$ on this copy of $\mm _{0,5} $.  One has
$$R(m, [j,k])  =     D(m) \cap P(j,k) \, \,\,\,\, {\text {, if  }}m \not \in  \{j,k\}  \, .   $$ On the other hand $$  D(j) \cap P(j,k) = \emptyset \, .$$
\paragraph{} 
\indent   On each curve $R(m, [j,k])$ there are three distinguished   points  ,
$$ L[(j,k),(i,n)]    \in  R(m, [j,k])  \, {\text {, if  }}  \{i,n\} \cap\{m,j,k\} =\emptyset \, .$$ 

The intersection of two del Pezzo components is the set
$$D(m) \cap   D(n)=  \{    L[(i,j),(k,l)] ,   L[(i,k),(j,l)]  ,
L[(i,l),(k,j)] \}_{\{m,n\} \cap\{i,j,k,l\} =\emptyset }  $$  
 
Finally  $$L[(i,j),(k,l)]=  D(m) \cap   D(n)  \cap P(i,j)\cap P(k,l)\, .$$

This is exactly the reason why $F(\mathcal S)$ is not a normal crossing surface, along these points there meet $4$ surfaces, and moreover the surfaces split in two couples 
which intersect locally only at the point.

\subsection{ {\label{sect: SEMISTABILE} } A semistable degeneration for the  smooth Fano surface} \paragraph{}
The  results of   \cite { GW}  are   based on the  study of the geometry    of the degeneration of cubic  threefolds to the singular Segre primal $\mathcal S $, we use the same specialization to understand  the topology  of the Fano surface.  An affine  pencil of cubic threefolds in $\mathbb P^4$ is  a  hypersurface  $\mathcal X  \subset \mathbb P^4 \times \mathbb A^1$ determined by  its  equation $  tE  +  G =0 $,  $E$ and $G$ are cubic polynomials.  The corresponding family of Fano surfaces, $\mathcal F \subset G(1,4) \times \mathbb A^1 $
has fibres $F_t$,  the surface of lines on $X_t$.
By  taking  $G$ to be the equation of $\mathcal S $ (up to  a change of coordinates one such is  $\, G:=\,\sum x_k^3+\sum_{i\ne j}x_i^2 x_j =0\,)$  $\mathcal F $ is a family of surfaces with central fibre $F(\mathcal S)\,.$  This  is not a semistable degeneration, because the central fibre has not normal crossing and also because $\mathcal F $ is singular. We  are concerned    with the singularities of  $\mathcal F$   just locally, near  $t= 0$.  

	\indent  It is known that the Fano surface of a cubic $3-$fold can be singular only along  the locus of lines which meet the singular points of the cubic $3-$fold.  As a consequence of this (by taking   further the polynomial   $E$  general enough,   so that $X_t$ does not pass through  the nodes of $\mathcal S$) we see that  our  threefold  $\mathcal F$ can be singular only on  the central fibre, and more precisely 
just along those points which represent lines on $ \mathcal S$ passing through the nodes.  Using Maple I have performed a computation akin to some which  can be  found in \cite{CG}:
\begin{proposition} \label{formula}
Near $t=0$ the pencil  $\mathcal F$  is singular exactly at the distinguished  points $L[(i,j),(k,l)]$  of  the central fibre $F(\mathcal S)$.
All those  points are ordinary quadratic singularities. \end{proposition} 

Our next step is to desingularize  $\mathcal F\,. $ One way is to blow up all the nodes, thus replacing each one of them with a non singular quadric surface.  We find more convenient to  choose instead  the local analytic process of replacing   every node  with a projective line, 
thus performing  $45$ small blow-ups over $\mathcal F$.
As it is well known there are  two possible choices for each  line. We take the one which has the effect to blow up the two planes 
containing  $L[(i,j),(k,l)]$ at that point, and on the contrary leaves unchanged  the  two del Pezzo surfaces through it. Our concern here  is in the topology of the Fano surface,
but later we need to appeal to the Clemens Schmid  exact sequence, and thus we   need to check that the process of blow up described takes place in the Kaehler category.
This is in fact the case, because   our process can be performed by blowing up in a sequence the smooth del Pezzo surfaces inside the larger smooth   ambient space on which the threefold lies.   Each blowing up leaves unchanged the base surface and also the threefold away from the double points on the surface, while at those double points  the node is blown up to   a projective line, which becomes in our case  an exceptional line on the proper transform  on the planes passing through it. The other del Pezzo surfaces passing through one of the nodes of the base surface become disjoint from it, meeting the exceptional line in a different point from the isomorphic transform of the chosen base surface.
\begin {definition}      $\widetilde {\mathcal F}   \to  {\mathcal F} $ is  the 45-small blow-up map, with the said  requirements .
\end {definition}
\begin {proposition}  Locally near $0$ on $\mathbb A^1$,    $\widetilde {\mathcal F}  \to  \mathbb A^1$ is a semistable degeneration of surfaces, namely 
 $\widetilde {\mathcal F}$ is non singular and the central fibre is union of smooth surfaces which have normal crossings. Recall that this says:
 \begin{enumerate}
  \item two components  are either disjoint, or     intersect transversally   along a smooth irreducible curve.
  \item three components meet each other at most in one point, and in that case analytically   like three planes  do.
  \item four components have  an empty intersection.
\end{enumerate}
\end {proposition} 

\begin {proof}  It is  an easy  consequence of   the combinatorics of  intersection   for the components of   $F(\mathcal S)$. 
\end {proof}

\subsubsection{ Combinatorics of the central fibre on $\widetilde {\mathcal F}$.   }\paragraph{}

The threefold  $\widetilde {\mathcal F}$ is the result of blowing up  $\widetilde {\mathcal F}$   in such a manner that    every  plane  $P(i,j)$  has been blown up at the $6$ points $L[(i,j),(k,l)]$ which are the intersection
of the four lines $R(m, [i,j])$:  
\begin {definition} $ $ \newline \begin{enumerate}
\item  $\widetilde {F_0}$ is the central fibre of $\widetilde {\mathcal F}$ \item $B(i,j) \to P(i,j)$  is the blow-up of  $P(i,j)$  at the $6$ points $L[(i,j),(k,l)],$ 
\item $ E[(i,j),(k,l)]=\emptyset  , {\text{ \,\, if \,}}  {\{i,j\} \cap\{k,l\} \not=\emptyset }  $
\item $E[(i,j),(k,l)]$ is the exceptional line over $L[(i,j),(k,l)]\, .$
\item $D(m) $   is  the component on $\widetilde {F_0}$ which maps  isomorphically to the surface by the same name on $F(\mathcal S)\,.$
\item $R(m, [m,k]) =\emptyset     .$
\item $R(m, [j,k])$      is  the curve on $ \widetilde {F_0}$ which maps  isomorphically to the curve  by the same name on $F(\mathcal S)\, .$
\item $N(m,[(m,j),(k,l)] ) :=\emptyset $ 
\item $N(m,[(i,j),(k,l)] ) := E[(i,j),(k,l)] \cap D(m) \subset \widetilde {F_0} $  
 \end{enumerate}\end {definition} 

With these conventions we can describe easily   the configuration of the central fibre: \begin{proposition} \label{decompongo}The  central fibre has normal crossing components:
$$ \widetilde {F_0}  =   (\bigcup_{i=1..6}  D(i) ) \, \bigcup \, \,( \bigcup_{ \{j,k\} \subset \{1..6\} }B(j,k)\,\, ) $$  their intersections are: \begin{itemize}
\item  $D(m) \cap   D(n) = \emptyset \, .$
\item  $B(m,n) \cap   B(k,l) = E[(m,n),(k,l)]$ .    
\item  $D(m) \cap   B(j,k) = R(m, [j,k])$.
\item  $D(m) \cap   B(k,l)\cap B(i,j   )= N(m,[(i,j),(k,l)] ) \, .$
\end{itemize}
\end{proposition}  

\subsubsection{Kneser graphs and the dual complex of $ \tilde {F_0}$ }  \paragraph{}
The intersection pattern of the components of $ \tilde {F_0}$ is codified by its  dual simplicial complex $\Lambda(\tilde {F_0})$,
its vertices correspond to the components, two vertices are joined iff the components meet, and a $2-$simplex is formed out of $3$ vertices iff the relative   components intersect   in a triple point. \paragraph{} By definition      the Kneser graph $KG(m, n)$, $m \geq  2n$,  has  vertex set  the
collection of all $n$-subsets of $[m]$, and two vertices are adjacent if and only if
they are disjoint as $n$-subsets. The incidence graph of the lines on the  Del Pezzo surface  is   $KG(5, 2)$, otherwise known as the Petersen graph. Similarly the combination  pattern of  the surfaces $B(m,n)$ determines  a    graph  $\Gamma$ with vertices  $v_{ (m,n)}$ and then   $\Gamma= KG(6,2) \, .$  \paragraph  {} The dual simplicial complex $\Lambda(  \tilde {F_0})$  has its $1-$skeleton isomorphic with    $KG(7,2)$ . The $6$  vertices  of type $[i,7]$ in $KG(7,2)$  correspond to the components $D(i)$, and the other vertices correspond to the components $B(m,n)$.   The  $2-$faces  to be added in order to pass from $KG(7,2)$  to  $\Lambda$   are  the triangles with a vertex of type  $[i,7]$. A  Gap computation,  which was  kindly shown to me  by  Prof. Soicher  \cite {SOI},    says 
\begin {proposition}   \label { bettisingolare} $\pi_1( \Lambda(\tilde {F_0})) \simeq  \mathbb Z^{\oplus 5 }$ and therefore $h_1( \Lambda(\tilde {F_0}))=5$  ,   $h_2( \Lambda(\tilde {F_0}))=10$.
\end  {proposition}

\section{ Some facts  on  the cohomology of $\tilde F_0$ and of $F$ .  } 
\subsection { The cohomology of $\tilde F_0$ . }\paragraph{}
We recall next  some fundamental  results, see  \cite {{GrS}}.
\subsubsection{ Deligne's  MHS on  a normal crossing variety}\paragraph{}
Let $X$  be a simple normal crossing variety,  so $X=\cup_{i} X_i$ where the  $X_i$ intersect locally like a union of coordinate  hyperplanes.  For $ I:=  \{ i_0 , \dots ,i_p \} , $  let $\,\, X_I  := X_{i_0} \cap \dots   \cap X_{i_p}$ and let $X^{[m] }= \sqcup _ {\{ | J | = m \} } X_J $  be the disjoint union of the iterated intersections of lenght $m$ .
The MHS on the cohomology of $X$ is  determined  by means of the  spectral sequence which abuts to it. One has  a double complex  $E_0^{p,q} =  A^q(X^{[p] })$ with vertical arrow  the usual de  Rham operator $d_0 = d$, while the horizontal arrow   $d_1 = \delta $, is induced by alternating restriction of forms, so taking $I:=  \{ i_0 , \dots ,i_{p+1} \}$ this is 
$$ (\delta \alpha)   ( X_I ) :=  \sum (-1)^j     \alpha ( X _ {  I \setminus \{ i_j\}  }  ) | X_I \,\, .$$
The first term of  the spectral sequence is $E_1^{p,q} = H^q( X^{[p]} )$ and then      $E_2^{p,q} $ is the middle cohomology of 
\begin{equation} \label {eq:E2}  H^q( X^{[p-1]} ) \xrightarrow{d_1}    H^q( X^{[p]} ) \xrightarrow{d_1}  H^q( X^{[p+1]} ) \end{equation}
The spectral sequence degenerates at $E_2$.
The weight filtration for $H^m (X)$ is $W_\ell := \bigoplus _{ s \leq \ell } E_2 ^{m-s , s }\, ,$ and the  Hodge filtration 
is the filtration induced by the usual Hodge structure on each factor in $E_1$. 
\subsubsection{Clemens-Schmid } \paragraph{}The Clemens-Schmid exact sequence for a semistable degeneration $\mathcal X \supset X_0$ is the  exact sequence of MHS's:

\begin{equation} \cdots  \to  H_{2n+2-m} (X_0)  \xrightarrow{j} H^{m} (X_0)\xrightarrow{c^{\ast}} H^{m} \xrightarrow{N} H^{m} \xrightarrow{k}  H_{2n-m} (X_0) \to \cdots  \end {equation}
here $N$ is the nilpotent monodromy operator and  $c^{\ast}$ as in ( \ref {controvariant}).
We take it in the form  
\begin{equation} \label{CSSURFACE} 0  \xrightarrow{j} H^{1} (\tilde {F_0})\xrightarrow{c^{\ast}} H^{1} \xrightarrow{N} H^{1} \xrightarrow{k}  H_{3} (\tilde {F_0}) \xrightarrow{j} H^{3}  (\tilde {F_0}) \xrightarrow{c^{\ast}} H^{3} \xrightarrow{N} H^{3} \xrightarrow{k}  H_{1}  (\tilde {F_0})  \xrightarrow{}  0 \end {equation}
so that  $ H^{m}  $    is the rational  cohomology of the non singular  Fano surface, equipped with the asymptotic  mixed Hodge structure. It is $h_3(F)= h_1(F) = 10 \, .$ \paragraph{} We need to understand some maps with range   $H_1(F)$, which by Poincare  duality   is  just $ H^3(F)$. Our first aim is to see    that  $H^{3} (\tilde F_0)\xrightarrow{c^{\ast}} H^{3}=H_1$  is  an injection with image an  isotropic subspace. We  use  here  the    intersection form  on $H_1(F)= H^{3} (F)$  which is  induced by    the  isomorphism  $H_{1} (F) \simeq H_3 (X)\simeq  H_1(J)$,   $X$ being  the  smooth   cubic 3-fold and $J$ its polarized intermediate Jacobian, see the introduction in  \cite{CG}. 
\subsection{    $H^3 (\tilde F_0) $. } \paragraph{}
The irreducible components of $\tilde F_0$  have dimension $2$, so  in the spectral sequence it is   $E_1 ^{p , q } = 0 \, $ for  $p \geq 3$ and moreover one has    $  E_1 ^{2 , q } = 0  $ when $ q \geq 1\, .$ In our setting the components   are rational surfaces which gives   $E_1 ^{0 , 3 } = 0$ and then    $H^3 (\tilde F_0)= E_2 ^{1 , 2 }\, $, i.e.  the following is an exact sequence:
\begin{equation}  \label  {coboundary}H^2( \tilde F^{[0]} ) \xrightarrow{d_1^{0,2}}    H^2( \tilde F^{[1]} ) \xrightarrow{}  H^3 (\tilde F_0)  \xrightarrow{} 0 \end {equation} 
Consider a component   $S$ of  $\tilde F ^{[0]} $;  in $S$ there is the normal crossing divisor  $C$  which is the intersection of $S$  with the remaining  components. As $S$ is a rational surface   then
the long exact sequence of relative cohomology for the couple $(S,C)$ reads 
\begin{equation} \cdots  \to  H^{2 } (S)  \xrightarrow{i_{C}^{\ast}} H^{2 } (C )\xrightarrow{} H^{3}( S,C)  \xrightarrow{} 0 \, \, .  \end {equation}
It is clear   that ${i_{C}^{\ast}} $ coincides, up to sign,  with the restriction of $d_1^{0,2}$ to $S$.  \paragraph{}  Recall now the notations from  (\ref {omologiauno}),     then the factorization (\ref {fattorizzo}) reads  
\begin{equation}   H_1 ( S \setminus C)= H^{3}( S,C)   \to  E_2 ^{1 , 2 } =  H^3 (\tilde F_0) \to 
H^{3} (F) =H_1 (F)  \end {equation}
\begin{proposition} 
\label{lem:iniettivita}  $ $  \begin {enumerate}
\item \label {primo} $dim_{\mathbb Q} H^3 (\tilde F_0 , \mathbb Q)  =5 \,.$
\item  \label {secondo} $H^{3} (\tilde F_0)\xrightarrow{c^{\ast}} H^{3}=H_1$ is an inclusion. 
\item $H^{3} (\tilde F_0)$ is a totally isotropic subspace of $H_1$ .\item  $H^3 (J) $ is the dual space  of  $\wedge ^3 H_1(F) $. Under their   pairing   the  unprimitive cohomology  $\theta \cdot H^1(J)$   vanishes  on  the subspace $\wedge ^3     H^{3} (\tilde F_0)$,  and therefore for any  nontrivial element   in  $\wedge ^3    H^{3} (\tilde F_0)$ there is    primitive  class inside  $H^3 (J) $ which does not vanish on said element.  \end  {enumerate}
\end{proposition} 
\begin{proof} 
The dual graph of the singular fibre is  $\Lambda(\tilde {F_0})$, whose betti numbers we have found in \ref { bettisingolare}. In our setting  one has    $W_0  H^1 (\tilde F_0)= W_1  H^1 (\tilde F_0)= H^1 (\tilde F_0) $, and therefore $dim H^1 (\tilde F_0) =  h^1 (\Lambda(\tilde {F_0}) )= 5$ .    Now Lemma (2.7.4) in  \cite {Pe} yields that it is also $dim E_2 ^{1 , 2} = 5,$  which is the first item, because $H^3 (\tilde F_0)= E_2 ^{1 , 2 }.$    Statement (\ref{secondo})   descends easily    from     exactness  of  \ref {CSSURFACE}.  Items ($3$) and  ($4$)  are  proved by  recalling  the identifications   $H^3 (F) \simeq H^3 (X)\simeq H_3(X)$.  By exactness $H^3 (\tilde F_0 , \mathbb Q)= Ker  N $, which  is the module of invariant cycles for $H_3(X)$ and it follows from  (\ref{primo})   that it coincides  then with the modulo of vanishing cycles, which space is   totally isotropic for the intersection pairing, cf. (3.3) in \cite {Cl}. The proof of item $4$ is routine, because to give  the polarization $\theta$ on $J $  is equivalent to give  to   the intersection coupling $Q$ on $H_1(J )=H_3(X)$ by setting  $( \theta, \alpha \wedge \beta ) = Q (\alpha, \beta)$. \end{proof}
\begin{proposition}  \label{riniettivita} (a)   For    a   component $S$ of type   $B(i,j)$   the  arrow   $$H_1(S \setminus C) = H^{3}( S,C)   \rightarrow  E_2 ^{1 , 2 } =  H^3 (\tilde F_0)$$  is an inclusion.  (b)  When  the component  $S$  is  a Del Pezzo surface the same map gives  an isomorphism:   $$ H_1(M_{0,5})  \simeq  H^3 (\tilde F_0)$$  \end{proposition} 
\begin{proof} 
To prove this directly by means of the spectral sequence  requires dealing with large matrices, although of a   simple type,  which  seems to be a  rather complex  task.  We choose   instead  to use Prop. \ref {irriducibile}. The kernel of $H_1(M_{0,5}) \to H^3 (\tilde F_0)$ is   invariant  under the action of  $S_5$,  if this  map is not injective then it must be the  trivial  morphism.  In this case,  again by symmetry, the arrows above  are then trivial for all the components which are Del Pezzo surfaces. On the other hand, having fixed $i$ and $j$,  recall that then $S_4 \subset S_6 $ acts on the surface $B(i,j)$.  By the same kind of argument as the one given  for   Prop. \ref {irriducibile} we see that the representation of $S_4$ on $H^3(B(i,j), \partial B(i,j))$  is the standard representation,  and therefore if  $H^3(B(i,j), \partial B(i,j))   \rightarrow    H^3 (\tilde F_0)$ is not an injective morphism then it vanishes identically. Again by symmetry, if this was the case then   the maps   $H^3(B(i,j), \partial B(i,j))   \rightarrow    H^3 (\tilde F_0)$ vanish  for all surfaces $B(i,j)$.   Now one can check  that a Del Pezzo component and one component  $B(i,j)$ which intersect  contribute  at least  a common generator to $H^3 (\tilde F_0)$, the class of their intersection.  In this way, if the map  to $  H^3 (\tilde F_0)$ was not injective for one of this surfaces  we know that it vanishes identically,   and then also the map for  the other surface should vanish on the common generator, hence identically.   Looking at the boundary morphism in   (\ref  {coboundary})  we see that  it is then    $H^3 (\tilde F_0 , \mathbb Q) =0$, which is a contradiction. \end {proof} 

\section{On the second homotopy group of the Fano surface} 
\subsection{Exact Sequences}
\paragraph{}   For an abelian variety $\pi_1 (A) = H_1(A,\mathbb Z)$ and $\pi_i (A) =0 , i \geq 2\, ,$ and therefore 
the   Albanese map $ F \to A \, ,$ which is for the Fano surface    an embedding,    yields the  diagram:
\begin{equation}
\label {eq:BUBU}
\xymatrixcolsep{10pt}
\xymatrix{
         \pi_3(F) \ar[r]  \ar[d]  & 0 \ar[d]  \ar[r] & \pi_3(A,F)  \ar[d]^{h_3}  \ar[r]   &  \pi_2 (F)  \ar[d]  \ar[r] &  0 \ar[d]  \ar[r] &  \pi_2(A,F)  \ar[d]^{h_2}  \ar[r] & \pi_1(F)\ar[d]   \ar[r]  & \pi_1(A) \ar[d] ^{=}  \\
   H_3(F)   \ar[r] & H_3(A) \ar[r]& H_3(A,F)   \ar[r] &H_2(F)   \ar[r] & H_2(A) \ar[r] &  H_2(A,F)  \ar[r] &  H_1(F)    \ar[r] ^{=}&      H_1(A)    }
\end{equation}

We have     $H_1(F,\mathbb Z)  =  H_1(A,\mathbb Z) \simeq \mathbb Z^{10} $, \cite{CG} Corollary 9.5. The  generalized  Hurewicz theorem, cf.  \cite{Sp},   gives the surjectivity of $h_2$, because   $\pi_1(A,F) $ is  trivial, and therefore     $h_2$     is   the quotient map $$ [\pi_1(F),\pi_1(F)] \to    [\pi_1(F),\pi_1(F)] /[ [\pi_1(F),\pi_1(F)], \pi_1(F)] \, . $$ The lower central series  of a group $G$ being    defined inductively as:    $\Gamma_1 G=G$ and  $\Gamma_{k+1} G= [\Gamma_{k} G,G]\, ,$ we see   that   $  H_2(A,F)$  is the second graded piece of the lower central series  for $\pi_1(F)\, .$     One has, see   \cite{collinoLNM}, \begin {equation}\label {DDDD}  H_2(A,\mathbb Z)/H_2(F,\mathbb Z))\simeq \mathbb Z/ 2\mathbb Z \, \,  .\end {equation}
The second   lower quotient  of the fundamental group     has been  revisited recently  by    \cite{BV}.
\label {errata } \begin{remark} Note that   $\pi_1(F)\,$  is   a central extension of  $\mathbb Z^{\oplus 10 } $. In my paper \cite{collinoLNM}   I  used without   proof the fact that  on the second symmetric product $C^{(2)} $ of a hyperelliptic curve $C$ of genus $g$  the complement of the line   $L$  given by the    standard  $g^1_2$  is such that  $\pi ( C^{(2)}   \setminus  L )  $  is  a  central extension of  $ \mathbb Z^{2g} $  by  $ \mathbb Z/2$.  A proof  of this  can be found  in  \cite{C2}.  \end{remark} 
\subsection{  { ${ \mathbf \pi_2(F)  } $ is not trivial}}    \paragraph{}
We see from diagram \ref {eq:BUBU}   that  in order to prove that  $ \pi_2(F)  = \pi_3(A,F)  $ is not trivial it is enough to determine an element which can be computed not to be  in the kernel of the Hurewicz map $h_3$.   \newline  To this aim 
we use the  inclusion  $B^0 \subset  F $  where  $B^0$ is  homotopically equivalent to  the complement of four   general lines in  the projective  plane, cf.      \ref {decompongo} above . \paragraph{}
According to \cite {HA} up to homotopy $B^0$ is the union of the  $3$ copies of   $C^* \times C^*$ which  inside $C^* \times C^*\times C^*$ are given  by the condition that one of the coordinates is $1$ and therefore    
$B^0$  is equivalent to 
the $2$-skeleton of the real $3$-torus $T^3$. 
One has  $$\pi_1 (B^0) =H_1 (B^0)=\mathbb Z^3 \, , \qquad  H_2 (B^0) =  \mathbb Z^3 \, ,\qquad  H_m (B^0) = 0 , m \geq 3 \,\, .$$
Consider the  universal covers $\tilde B^0$ of $B^0$ and $\widetilde{T^3}$ of $T^3$:
   \begin  {proposition}  \label {fundB0}    \cite {HA} $$\pi_2(B^0)=H_2( \tilde B^0)  { \text { is a free }} \mathbb Z {\mathbb Z^3}\,{\text{module }} \, ,$$ generated by the
boundary of
a cubical $3$-cell from the standard decomposition of
$\widetilde{T^3}=\R^3$. \end {proposition}  Our result is: 
\begin{theorem} $ \pi_2(F)   $ is not of torsion.
\begin{proof}    We work up to homotopy and use the map  $B^0 \to  \mathbb C^{\ast \,\, 3} \, .$ \paragraph{}
 The Albanese variety $A$ of $F$ is isomorphic to    the intermediate jacobian $J$ of the corresponding cubic 3-fold, hence     $A=J$   is  a principally polarized abelian variety, of dimension $5\, .$    
The inclusions  $B^0 \hookrightarrow F \hookrightarrow  J$ induce  the maps $$\mathbb Z^3=\pi_1 (B^0)= H_1(B^0, \mathbb Z) \rightarrow  H_1(F, \mathbb Z) =  H_1(J, \mathbb Z) =\pi_1 (J)=\mathbb Z^{10},  $$ and therefore one has    a continuous map $$q: \mathbb C^{\ast \,\, 3} \to J $$  which, up to homotopy,  is a map of couples  $q:(\mathbb C^{\ast \,\, 3} ,B^0 ) \to (J,  B^0)\,\, .$ 
There is a  commutative diagram:  {\begin{equation}\xymatrixcolsep{0.05  pt} \xymatrixrowsep{6  pt}
\xymatrix{
\cdots\to  &   0 \ar'[d][dd]    \ar[rd]  \ar[rr] &  &   \pi_3(\mathbb C^ {\ast\, 3},B^0)  \ar'[d] [dd] \ar[rd]^{ h_3  }  \ar[rr]  &  &  \pi_2 (B^0)  \ar'[d][dd] \ar[rd] \ar[rr]  &  & 0  \ar'[d][dd] \ar[rd]\\
0 =  H_3(B^0) \ar[dd]  \ar[rr] &&   H_3(\mathbb C^ {\ast\, 3})= \mathbb Z \ar[dd] \ar[rr]& & H_3(\mathbb C^ {\ast\, 3},B^0)    \ar[dd]  \ar[rr]  &  &  H_2(B^0)= \mathbb Z^ 3  \ar[dd] \ar[rr] &  & H_2(\mathbb C^ {\ast\, 3})= \mathbb Z^ 3  \ar[dd]& \\  \cdots\to  &  0   \ar'[r][rr]  \ar[rd] && \pi_3(J, F)  \ar'[r][rr]  \ar[rd]^{\, h  } & &  \pi_2(F)   \ar'[r][rr]  \ar[rd] & &0  \ar[rd]\\   H_3(F)   \ar[rr]    &  &  H_3(J) \ar[rr] && H_3(J,F)  \ar[rr]^{0} &  &  H_2(F)  \ar[rr] &  & H_2(J) }
\end{equation}
It is   $ H_3(\mathbb C^{\ast \,\, 3},\mathbb Z) \simeq  H_3(\mathbb C^{\ast \,\, 3},B^0) $   because $H_2(B^0,\mathbb Z)\simeq H_2( \mathbb C^{\ast \,\, 3} ,\mathbb Z)\, .$ 
 The generalized  theorem of  Hurewicz   for the couple $(\mathbb C^{\ast \,\, 3},  B^0)$ says that  the map  $h_3:   \pi_3(\mathbb C^{\ast \,\, 3},B^0) \to  H_3(\mathbb C^{\ast \,\, 3},B^0) $ is surjective and then   the generator $\zeta$  in $ H_3(\mathbb C^{\ast \,\, 3}) $  can be used to detect  the existence of a non torsion element    $z \in \pi_2(B^0)\, .$
 We   map $z$ to $\pi_2(F)$ i.e. to $\pi_3(J,F)  $ then to an element $\bar z  \in  H_3(J,F)  $.  By commutativity   $\bar z$  is the image of   $q_{\ast} (\zeta)  \in  H_3(J)  $, and then  $\bar z \neq 0 $ if  $q_{\ast} (\zeta) $ does not come from $ H_3(F)$.  The results given  above  in \ref {lem:iniettivita} and   \ref {riniettivita}   imply the existence  of  a  primitive element  in $ H^3(J)$ which acts non trivially on $q_{\ast} (\zeta)$. This   is precisely what we need to know, because the  kernel of  the restriction map  $ H^3(J) \to  H^3(F) $ is formed  by the primitive  classes,  cf.  (0.9) in  \cite{CG}.
It follows that  $\pi_2(F)$ is not   of torsion. } \end {proof} \end{theorem}

\subsection*{Acknowledgements}

I am   grateful to  L.H. Soicher for his 
kind help with Gap.
 
\end{document}